\documentclass{article}
\usepackage{}

\usepackage{amssymb,amsmath}
\usepackage{latexsym}
\usepackage[dvips]{graphicx}
\usepackage{graphicx,color}
\usepackage{hyperref}

\let\cal=\mathcal
\newtheorem{theo}{Theorem}[section]

\newtheorem{lemma}[theo]{Lemma}

\newcommand{\qed}{\hspace*{\fill} \rule{7pt}{7pt}}

\begin{document}

\title{On a conjecture of Tokushige for cross-$t$-intersecting families\thanks{E-mail addresses: huajunzhang@usx.edu.cn (H. Zhang),
wu@hunnu.edu.cn (B. Wu, corresponding author)}}

\author{
Huajun Zhang $^{\dag}$,
 Biao Wu $^{\ddag}$\\[2ex]
{\small $^{\ddag}$ Department of Mathematics, Shaoxing University} \\
{\small Shaoxing, Zhejiang, 321004, P.R. China }\\
{\small $^{\dag}$ MOE-LCSM,  School of Mathematics and Statistics, Hunan Normal University} \\
{\small Changsha, Hunan, 410081, P.R. China } }

\maketitle

\begin{abstract}
Two families of sets $\mathcal{A}$ and $\mathcal{B}$ are called cross-$t$-intersecting if $|A\cap B|\ge t$ for all $A\in \mathcal{A}$,  $B\in \mathcal{B}$.
An active problem in extremal set theory is to determine the maximum product of sizes of cross-$t$-intersecting families. This incorporates the classical Erd\H{o}s--Ko--Rado (EKR) problem.
In the present paper, we prove
that if $\mathcal{A}$ and $\mathcal{B}$ are cross-$t$-intersecting families of $\binom {[n]}k$ with $k\ge t\ge 3$ and $n\ge (t+1)(k-t+1)$, then $|\mathcal{A}||\mathcal{B}|\le {\binom{n-t}{k-t}}^2$; moreover, if $n>(t+1)(k-t+1)$, then equality holds if and only if $\mathcal{A}=\mathcal{B}$  is a maximum $t$-intersecting subfamily of $\binom{[n]}{k}$. This confirms a conjecture of Tokushige for $t\ge 3$.

\end{abstract}


\section{Introduction}

Let $n$, $k$  and $t$ be positive integers with $t\leq k\leq n$.  We write $[n]:=\{1,\ldots,n\}$ and denote its power set by $2^{[n]}$. A nonempty subset $\mathcal{F}$ of $2^{[n]}$ is called {\it $k$-uniform} if its elements all have size $k$.  We use $\binom{[n]}{k}$
to denote the set of all $k$-element subsets of $[n]$. We call $\mathcal{F}$ a {\it $t$-intersecting family} if $|A\cap B|\ge t$ for all $A,B\in \mathcal{F}$. A $1$-intersecting family is simply called an {\it intersecting} family. The Erd\H{o}s--Ko--Rado Theorem is one of the central results in extremal combinatorics.

\begin{theo}[Erd\H{o}s--Ko--Rado Theorem \cite{EKR1961}]\label{EKR1961}
Let $t,k,n$ be positive integers such that $t\leq k\leq n$. Suppose that $\mathcal F$ is a $t$-intersecting subfamily of $\binom{[n]}{k}$. Then for $n\geq n_0(k,t)$, \[|\mathcal{F}|\leq \binom{n-t}{k-t}. \]
\end{theo}

Let $k,t$ be fixed positive integers such that $t\le k$, and let $N_0(k,t)$ be the smallest possible value of $n_0(k,t)$ in Theorem \ref{EKR1961}. In the case $t=1$, it is shown in \cite{EKR1961} that $N_0(k,1)=2k$. For $t\ge 1$, we have $N_0(k,t)=(t+1)(k-t+1)$, which was established in \cite{Fra1976} for $t\ge 15$ and in \cite{Wil1984} for all $t$.

Let $r$ be an integer such that $0\le r\le \frac {n-t} 2$. We define the following $k$-uniform family
\[
  \mathcal{F}(n, k, t, r):=\left\{A\in \binom{[n]}{k}: |A\cap [t+2r]|\ge t+r\right\}.
\]
It is a $t$-intersecting family, since for its elements $A,B$ we have $|A\cap B|\ge |A\cap [t+2r]|+|B\cap[t+2r]|-(t+2r)\ge t$. Frankl \cite{Fra1976} conjectured that if $\mathcal{F}$ is a $t$-intersecting subfamily of $\binom {[n]}k$, then
\begin{equation}\label{eqn_Fbound}
 |\mathcal{F}|\le \max_{0\le r \le (n-t)/2} |\mathcal{F}(n, k, t, r)|.
\end{equation}
This conjecture was proved partially by Frankl and F\"uredi \cite{FF}, and then settled completely by Ahlswede and Khachatrian \cite{AK}.

\begin{theo}[Ahlswede and Khachatrian \cite{AK}]
Let $n,k,t$ be positive integers such that $t\le k\le n$ and $2k-t+1\le n\le (t+1)(k-t+1)$. Let $\mathcal F$ be a $t$-intersecting $k$-uniform family of subsets of $[n]$.
\begin{itemize}
\item [\rm(i)] If $(k-t+1)(2+\frac{t-1}{r+1})<n<(k-t+1)(2+\frac{t-1}{r})$ for some nonnegative integer $r$, then $|\mathcal F|\leq |\mathcal{F}(n, k, t, r)|$, and  equality holds if and only if $\mathcal{F}=\mathcal{F}(n, k, t, r)$ up to isomorphism.
\item [\rm(ii)]If $(k-t+1)(2+\frac{t-1}{r+1})=n$ for some nonnegative integer $r$, then $|\mathcal F|\leq |\mathcal{F}(n, k, t, r)|=|\mathcal{F}(n, k, t, r+1)|$, and equality holds if and only if $\mathcal{F}=\mathcal{F}(n, k, t, r)$ or $\mathcal{F}(n, k, t, r+1)$ up to isomorphism.
  \end{itemize}
\end{theo}

Let $\mathcal{A}$ and $\mathcal{B}$ be two families of subsets of $[n]$. They are called {\it cross-$t$-intersecting} if $|A\cap B|\ge t$ for all $A\in \mathcal{A}$ and $B\in \mathcal{B}$.
In the case $t=1$, we simply say that $\mathcal{A}$ and $\mathcal{B}$ are {\it cross-intersecting}.
The cross-$t$-intersecting property is a natural extension of the $t$-intersecting property, since the two properties coincide in the case $\mathcal{A}=\mathcal{B}$. Let $\mathcal A$ and $\mathcal B$ be cross-$t$-intersecting subfamilies of $\binom{[n]}{k}$.
We say that they are \emph{maximal} if there are no cross-$t$-intersecting families $\mathcal A_1$ and $\mathcal B_1$ of $\binom{[n]}{k}$ such that $\mathcal A\subseteq \mathcal A_1$,  $\mathcal B\subseteq\mathcal B_1$ and $|\mathcal A|+|\mathcal B|<|\mathcal A_1|+|\mathcal B_1|$.

The study of possible maximum sizes of pairs of cross-intersecting families is one of the central problems of extremal set theory.
In 1986, Pyber \cite{Pyber} generalized the Erd\H{o}s--Ko--Rado Theorem to the cross-intersecting case.
\begin{theo} [Pyber \cite{Pyber}] \label{Pyber}
Let $n,k,\ell$ be positive integers such that $\ell\le k\le n$. Suppose that $\mathcal{A}\subseteq \binom {[n]}k$ and $\mathcal{B} \subseteq \binom {[n]}{\ell}$ are two cross-intersecting  families.
\begin{enumerate}
\item[{\rm (1)}] If $n\ge 2k+\ell-2$, then $ |\mathcal{A}||\mathcal{B}|\leq \binom{n-1}{k-1}\binom{n-1}{\ell-1}$;
\item[{\rm (2)}] If $k=\ell$ and $n\ge 2k$ then $ |\mathcal{A}||\mathcal{B}|\leq \binom{n-1}{k-1}^2$.
\end{enumerate}
\end{theo}

 For $k>\ell$, the lower bound on $n$ in (1) of Theorem \ref{Pyber} turns out to be not best possible. In 1989, Matsumoto and Tokushige \cite{MT1989} obtained the following sharper result.

\begin{theo}[Matsumoto and Tokushige \cite{MT1989}] \label{theoMT}
Let $n,k,\ell$ be positive integers such that $n\ge 2k\ge 2\ell$.
If $\mathcal{A}\subseteq {\binom{[n]}{k}}$ and $\mathcal{B}\subseteq {\binom{[n]}{\ell}}$ are cross-intersecting, then
$ |\mathcal{A}||\mathcal{B}|\leq {\binom{n-1}{k-1}}{\binom{n-1}{\ell-1}}$. Moreover, equality holds if and only if $\mathcal{A}=\{A\in {\binom{[n]}{k}}: x\in A\}$ and $\mathcal{B}=\{B\in \binom{[n]}{\ell}: x\in B\}$ for some fixed element $x\in[n]$.
\end{theo}

In 2013, Tokushige \cite{Tokushige2013} proved the following analogous result for cross-$t$-intersecting families.

\begin{theo} [Tokushige \cite{Tokushige2013}] \label{theoTokushige}
Let $n,k,t$ be positive integers such that $t\le k\le n$ and $\frac{k}{n}<1-\frac{1}{\sqrt[t]{2}}$. If $\mathcal{A}\subseteq {\binom{[n]}{k}}$ and $\mathcal{B}\subseteq {\binom{[n]}{k}}$ are cross-$t$-intersecting, then $ |\mathcal{A}||\mathcal{B}|\leq \binom{n-t}{ k-t}^2.$ Moreover, equality holds if and only if $\mathcal{A}=\mathcal{B}=\{F\in {\binom{[n]}{k}}: T\subset F\}$ for some $t$-subset $T$ of $[n]$.
\end{theo}

In the same paper \cite{Tokushige2013}, Tokushige conjectured that the lower bound on $n$ in Theorem \ref{theoTokushige} can be improved to $n \ge (t+1)(k-t+1)$. In 2014, Frankl, Lee,  Siggers and Tokushige \cite{FLST} verified the conjecture under the stronger assumptions that $t\ge 14$ and $n \ge (t+1)k$. In the same year, Borg \cite{Borg2014} independently confirmed the conjecture for large $n$.
 In 2016, Borg \cite{Borg2016} presented a analogous result for two more general families, yielding Tokushige's conjecture for general $t$, but for a smaller range of $n$.
 In the present paper, we confirm Tokushige's conjecture for all $t\ge 3$. The following is our main result.

\begin{theo} \label{maintheo}
Let $n,k,t$ be positive integers such that $3\le t\le k\le n$ and $n \ge (t+1)(k-t+1)$.
If $\mathcal{A}\subseteq {\binom{[n]}{k}}$ and $\mathcal{B}\subseteq {\binom{[n]}{k}}$ are cross-$t$-intersecting, then $ |\mathcal{A}||\mathcal{B}|\leq \binom{n-t}{ k-t}^2.$ Moreover, equality holds if and only if $\mathcal{A}=\mathcal{B}$ is a maximum $t$-intersecting subfamily of $\binom{[n]}{k}$.
\end{theo}

The methods used in this paper are the shift operator method \cite{EKR1961} and the generating set method \cite{AK}.
They can be used to treat some other extremal set problems, such as determining the maximum (product of) sizes of (cross-)intersecting multi-part families, and the maximum sum of sizes of cross-intersecting families.
Their definitions and basic properties will be given in the next section. In Section 3, we will give an essential inequality. In section 4, we will give  the complete proof of our main result.

\section{Preliminaries}

Let $\mathcal A$ be a family consisting of $k$-element subsets of $[n]$.  For $i,j\in [n]$  and $A\in\mathcal A$, define
\[
  \delta_{ij}(A)=\begin{cases}(A\setminus\{j\})\cup\{i\}, & \hbox{if $j\in A,\ i\not\in A, (A\setminus\{j\})\cup\{i\}\not\in \mathcal{A}$;}\\A, & \hbox{otherwise,}\end{cases}
\]
and set $\Delta_{ij}(\mathcal A)=\{\delta_{ij}(A):A\in\mathcal A\}$ correspondingly. The procedure to obtain $\Delta_{ij}(\mathcal A)$ from $\mathcal A$ is called the shift operation, and was first introduced in
\cite{EKR1961} (see also \cite{Fra1987}). We observe that $\Delta_{ij}(\mathcal A)$ has the same cardinality as $\mathcal{A}$ and is also $k$-uniform. We say that $\mathcal A$ is \emph{left-compressed} if $\Delta_{ij}(\mathcal A)=\mathcal A$ for all $(i,j)$ pairs such that $1\leq i<j\leq n$. It is well known that if $\mathcal A$  and $\mathcal B$ are cross-$t$-intersecting subfamilies of $\binom{[n]}k$, then so are $\Delta_{ij}(\mathcal A)$ and $\Delta_{ij}(\mathcal B)$.

For a subset $E$ of $[n]$, we set $\mathcal{U}(E)=\{A\subseteq [n]: E\subseteq A \}$ and $s^+(E)=\max\{i: i\in E\}$. For $\mathcal{E}\subseteq 2^{[n]}$, we set $\mathcal{U}(\mathcal{E})=\cup_{E\in \mathcal{E}}\mathcal{U}(E)$ and $s^+(\mathcal E)=\max\{s^+(E): E\in\mathcal E\}$ correspondingly. Let $\mathcal{F}$ be a family of $k$-element subsets of $[n]$. We say that a family  $g(\mathcal{F})\subseteq \bigcup_{i\leq k}\binom{[n]}{i}$   is a \emph{generating set} of $\mathcal{F}$ if $\mathcal{U}(g(\mathcal F))\cap  \binom{[n]}{k}=\mathcal{F}$. It is clear that $\mathcal{F}$ is a generating set of itself. The set of all generating sets of $\mathcal{F}$ forms a nonempty set which we denote by $G(\mathcal{F})$. The notion of generating sets of a $k$-uniform family was first introduced in \cite{AK}. For $g(\mathcal F)\in G(\mathcal F)$, let $g_*(\mathcal F)$ be the set of all minimal (in the sense of
set-theoretical inclusion) elements of $g(\mathcal F)$. Set $G_*(\mathcal F)=\{g(\mathcal F)\in G(\mathcal F): g(\mathcal F)=g_*(\mathcal F)\}$. Notice that $\mathcal F\in G_*(\mathcal F)$. For $g(\mathcal F)\in G_*(\mathcal F)$, let $\ell=s^+(g(\mathcal F))$,  set $g^*(\mathcal F)=\{E\in g(\mathcal F): \ell\in E\}$, $g^*_i(\mathcal F)=\{E\in g^*(\mathcal F): |E|=i\}$  and $g^*_i(\mathcal F)'=\{E\setminus\{\ell\} : E\in g^*_i(\mathcal F)\}$ for $1\leq i\leq \ell$. From \cite{AK}, we know that the generating sets have the following properties.

\begin{lemma}\label{gtf} Let $\mathcal F$ be a left-compressed $t$-intersecting subfamily of $\binom{[n]}{k}$, $g(\mathcal F)\in G_*(\mathcal F)$ and  $\ell=s^+(g(\mathcal F))$. Then the following statements hold.

  {\rm (i)} If $n>2k-t$, then $|E_1\cap E_2|\geq t$ for all $E_1,E_2\in g(\mathcal F)$.

  {\rm (ii)} For $1\leq i<j\leq \ell$  and $E\in g(\mathcal F)$, we have either $\delta_{ij}(E)\in g(\mathcal F)$ or $F\subset \delta_{ij}(E)$ for some $F\in g(\mathcal F)$.

  {\rm(iii)} $\mathcal F$ is a disjoint union \[\mathcal F=\bigcup_{E\in g(\mathcal F)}\mathcal D(E), \mbox{ where }
  \mathcal D(E)=\left\{B\in\binom{[n]}{k}: B\cap [s^+(E)]=E\right\}.\]

  {\rm(iv)} If $\mathcal F$ is maximal, then for any $E_1, E_2\in g^*(\mathcal F)$ with $|E_1\cap E_2|=t$, necessarily $|E_1|+|E_2|=\ell+t$ and $E_1\cup E_2=[\ell]$. Furthermore, if $g^*_i(\mathcal F)\neq \emptyset$, then $g^*_{\ell+t-i}(\mathcal F)\neq \emptyset$ and for any $E_1\in g^*_i(\mathcal F)$, there exists $E_2\in g^*_{\ell+t-i}(\mathcal F)$ with $|E_1\cap E_2|=t$ and $E_1\cup E_2=[\ell]$.

{\rm(v)} If  $g^*_i(\mathcal F)\neq\emptyset$, then $\mathcal F_1=\mathcal F\cup\mathcal D(g^{*}_i(\mathcal F)')\backslash \mathcal D(g^{*}_{\ell+t-i}(\mathcal F))$ is also a $t$-intersecting subfamily of $\binom{[n]}{k}$ with
$$|\mathcal F_1|=|\mathcal F|+|g^{*}_i(\mathcal F)|\binom{n-\ell}{k-i+1}-|g^{*}_{\ell+t-i}(\mathcal F)|\binom{n-\ell}{k+i-\ell-t}.$$
 \end{lemma}

Let $\mathcal A$ and $\mathcal B$ be two cross-$t$-intersecting subfamilies of  $\binom{[n]}{k}$.
 For $n>2k-t$, it is easy to find that $\mathcal A$  and $\mathcal B$ are cross-$t$-intersecting if and only if $g(\mathcal A)$ and $g(\mathcal B)$ are so for all $g(\mathcal A)\in G(\mathcal A)$ and $g(\mathcal B)\in G(\mathcal B)$.
By applying the shift operations if necessary, we may assume without loss of generality that both $\mathcal{A}$ and $\mathcal{B}$ are left-compressed. Similarly, on the generating sets of $\mathcal A$ and $\mathcal B$, we have the following result.

\begin{lemma}\label{lemmag1}
Let $\mathcal A$ and $\mathcal B$ be two maximal left-compressed cross-$t$-intersecting subfamilies of $\binom{[n]}{k}$, $g(\mathcal A)\in G(\mathcal A)$, $g(\mathcal B)\in G(\mathcal B)$ and $s=\max\{s^+(g(\mathcal A)), s^+(g(\mathcal B))\}$. Then the following statements hold.
\begin{itemize}
\item[\rm(i)]For $1\leq i<j\leq s$, $\cal F \in \{\cal A, \cal B\}$  and $E\in g(\mathcal F)$, we have either $\delta_{ij}(E)\in g(\mathcal F)$  or $F\subset \delta_{ij}(E)$ for some $F\in g(\mathcal F)$.
    \item [\rm(ii)] For $t\leq i\leq k$,  $g_i^*(\mathcal A)\neq\emptyset$ if and only if $g_{s+t-i}^*(\mathcal B)\neq\emptyset$. Furthermore, for each $E\in g_i^*(\mathcal A)$, there exists $F\in g_{s+t-i}^*(\mathcal B)$ such that $|E\cap F|=t$ and $E\cup F=[s]$.
        \item[\rm(iii)]If $g_i^*(\mathcal A)\neq\emptyset$, then $\mathcal A_1=\mathcal A\cup \mathcal{D}(g_i^*(\mathcal A)')$ and $\mathcal B_1=\mathcal B\setminus \mathcal D(g_{s+t-i}^*(\mathcal B))$ are also cross-$t$-intersecting families of $\binom{[n]}{k}$ with
            \[|\mathcal A_1|=|\mathcal A|+|g_i^*(\mathcal A)|\binom{n-s}{k-i+1} \ and \ |\mathcal B_1|=|\mathcal B|-|g_{s+t-i}^*(\mathcal B)|\binom{n-s}{k+i-s-t}. \]

 \end{itemize}
\end{lemma}
\begin{proof} (i) is obvious. Firstly,  we prove (ii). Suppose $g_i^*(\mathcal A)\neq\emptyset$ for some $t\leq i\leq k$. For any $E\in g_i^*(\mathcal A)$, set $E'=E\setminus\{s\}$.
If $|E'\cap F|\geq t$ for all $F\in g(\mathcal B)$, then $\mathcal A_1=\mathcal A\cup \mathcal D(E')$ and $\mathcal B$ are also cross-$t$-intersecting with $\mathcal A\subsetneq\mathcal A_1$, contradicting the maximality of $\mathcal A$ and $\mathcal B$. Therefore $|E'\cap F_0|\le t-1$ for some $F_0\in g(\mathcal B)$. However, $|E\cap F_0|\geq t$, so we have $|E\cap F_0|=t$ and $s\in E\cap F_0$.
Suppose $r\not\in E\cup F_0$ for some $r\in [s]$.
Then by (i), there exists $F_1\in g(\mathcal B)$ such that $F_1\subseteq \delta_{rs}(F_0)$.
However $|E\cap F_1|\le |E\cap \delta_{rs}(F_0)|\le t-1$, contradicting that $g(\mathcal A)$ and $g(\mathcal B)$ are cross-$t$-intersecting. Hence, $E\cup F_0=[s]$ and $|E\cap F_0|=t$, i.e, $|F_0|=s+t-i$ and $g_{s+t-i}^*(\mathcal B)\neq\emptyset$.

Secondly, we only prove the front part of (iii), since the latter part of (iii) is obvious. It suffices to prove that $g_i^*(\mathcal A)'$ and $g(\mathcal B)\setminus g_{s+t-i}^*(\mathcal B)$ are cross-$t$-intersecting.
Let $A\in g_i^*(\mathcal A)'$ and $B\in g(\mathcal B)\setminus g_{s+t-i}^*(\mathcal B)$.
If $s\notin B$, then $|B\cap A|=|B\cap (A\cup \{s\})|\ge t$ since $g(\mathcal A)$ and $g(\mathcal B)$ are cross-$t$-intersecting, and $A\cup \{s\}\in g(\mathcal A)$.
Otherwise $s\in B$. For the contrary suppose that $|A\cap B|\le t-1$. Then it must be $|(A\cup \{s\}) \cap B|= t$. By (i) and (ii), it is not difficult to see that $(A\cup \{s\}) \cup B=[s]$.
It follows that $B\in g_{s+t-i}^*(\mathcal B)$, contradicting to $B\in g(\mathcal B)\setminus g_{s+t-i}^*(\mathcal B)$.
Thus $|A\cap B|\ge t$. This completes the proof.
\qed
\end{proof}

In the following, we will introduce some results needed in Section \ref{section4}.

\begin{theo}[Wang and Zhang\cite{wz2}]\label{wzn}
Let $n$, $k$ and $t$ be positive integers with $n>2k-t$. If $\mathcal A$ and $\mathcal B$ are nonempty cross-$t$-intersecting subfamilies of $\binom{[n]}{k}
$, then \[|\mathcal A|+|\mathcal B|\leq \binom{n}{k}- \sum_{0\leq i\leq t-1}\binom k i \binom {n-k}{k-i}+1,\]
and equality holds if and only if one of the following holds:
\begin{itemize}
  \item [\rm(i)] up to isomorphism, either $\mathcal A=[k]$ and $\mathcal B=\{B\in \binom{[n]}{k}: |B\cap [k]|\geq t\}$ or $\mathcal B=[k]$ and $\mathcal A=\{A\in \binom{[n]}{k}: |A\cap [k]|\geq t\}$;
  \item [\rm(ii)]$(n,k,t)=(k+2,k,k-1)$, and up to isomorphism, $\mathcal A=\mathcal B=\binom{[k+1]}{k}$.
\end{itemize}
 \end{theo}
 \begin{theo}[Normalized matching property, cf. \cite{Sperner}]\label{nmm} Let $n$ and $j$ be two positive integers with $1\leq j\leq n-1$. For all non-empty $\mathcal F\subseteq\binom{[n]}{j}$,
 \[\frac{|\bigtriangledown(\mathcal F)|}{|\mathcal F|}\geq \frac{\binom{n}{j+1}}{\binom{n}{j}},\] where $\bigtriangledown(\mathcal F)=\{B\in \binom{[n]}{j+1}: \mbox{$A\subset B$ for some $A\in\mathcal F$}\}$. \end{theo}

\section{Some basic inequalities concerning the proofs}

In this section, we will prove a technical result, Theorem \ref{lemmakeyineq}, which is crucial to our proof of Theorem \ref{maintheo}.
The following Lemma is a very simple inequality, but the central idea of our proof.
\begin{lemma} \label{basefact}
Let $A,B,a,b$ be  positive real numbers. Then $(A+a)(B-b)<AB$ if and only if $\frac B {A+a} <\frac b a$.
\end{lemma}

Throughout this section we suppose that $n$, $k$, $s$, $t$ and $i$ are positive integers such that $(t+1)(k-t+1)\leq n$, $t+3\leq s\leq 2k-t$ and $\max\{t+1,s+t-k\}\leq i\leq \min\{k,\frac{s+t}{2}\}$, and define
\begin{equation}\label{H1234}
\left\{
\begin{aligned}
&S_1=i(n-k-s+i+1)+(s-i)(n-s+1)=s(n-s+1)-i(k-i), \\
&S_2=(s+t-i)(n-k-i+t+1)+(i-t)(n-s+1)=s(n-s+1)-(s+t-i)(k+i-s-t), \\
&T_1=i(n-k-s+i+1)+(s-i)(k-i+1),  \\
&T_2=(s+t-i)(n-k-i+t+1)+(i-t)(k-s-t+i+1).
\end{aligned}
\right.
\end{equation}
Since $t+3\leq s\leq 2k-t$, we have $t+2\le k$. Since $s+t-k\le i\le \frac{s+t}{2}$, we have $2i-t\le s\le k+i-t$. The remaining part of this section is devoted to the following key theorem.
\begin{theo}\label{lemmakeyineq}
Take notations as above and suppose $(s,i,t)\ne (6,4,3)$. Then we have
\begin{eqnarray}\label{ff1}
\frac{(k-i+1)(k+i+1-s-t)S_1 S_2}
{(n-s+1)^2 T_1 T_2}>\frac{\binom{n-s}{k-i}\binom{n-s}{k+i-s-t}}{\binom{n-s}{k-i+1}\binom{n-s}{k+i+1-s-t}}.
\end{eqnarray}
\end{theo}

We start with some technical lemmas.
\begin{lemma}\label{lem_prep1}
$S_1+S_2-T_1-T_2\geq s(k-i+1+k+i-s-t+1)$ if $(s,i,t)$ is not one of $(8,6,5),$ $(7,5,4)$, $(8,5,3)$, $(8,4,3)$, $(7,5,3)$, $(7,4,3)$ and $(6,4,3)$.
\end{lemma}
\begin{proof}
Set $f(s,i,t): = S_1+S_2-T_1-T_2-s(2k-s-t+2)$. After plugging in the expressions in \eqref{H1234}, we obtain
\[
f(s,i,t)=(s-i)(n+i-k-s)+(i-t)(n+t-k-i)-s(2k-s-t+2).
\]
We have $\frac{\partial f}{\partial i}(s,i,t)=2(s+t-2i)\geq 0$ (recall that $i\le \frac{s+t} 2$), and
\[
  \frac{\partial f}{\partial s}(s,i,t)=n+2i-3k+t-2\geq (t+1)(k-t+1)+3t-3k> (t-2)(k-t+1)>0
\]
by the assumptions that $n\ge (t+1)(k-t+1)$ and $i\ge t+1$. It follows that $f(s,i,t)$ is an increasing function in both $i$ and $s$. Recall that $s\ge t+3$.
If $t\ge 6$, then
\begin{eqnarray*}
f(s,i,t)\ge f(t+3,t+1,t)&=&2(n-k-2)+(n-k-1)-(t+3)(2k-2t-1)\\
&=& 3(n-k-1)-2(t+3)(k-t)+t+1
\geq(t-6)(k-t)+t+1>0.
\end{eqnarray*}
It remains to consider the case $t\le 5$. If $t\in\{4,5\}$ and $s=t+3$, then we deduce from $t+1\le i\le\frac{s+t}{2}$ that $(s,i,t)$ is one of $(8,6,5)$ and $(7,5,4)$, which have been excluded from consideration.
If $t\in \{4,5\}$ and $s\ge t+4$, then
\begin{eqnarray*}
f(s,i,t)\ge f(t+4,t+1,t)&=&3(n-k-3)+(n-k-1)-(t+4)(2k-2t-2)\\ &=&4(n-k-1)-2(t+4)(k-t)+2t+2
\geq2(t-4)(k-t)+2t+2>0.
\end{eqnarray*}
Finally, suppose that $t=3$. If $s\le 8$, then we deduce from $4=t+1\le i\le\frac{s+t}{2}$ that $(s,i,t)$ is one of $(8,5,3)$, $(8,4,3)$, $(7,5,3)$, $(7,4,3)$ and  $(6,4,3)$, which have been excluded from consideration. We thus have $s\ge 9$ and correspondingly
\begin{eqnarray*}
f(s,i,t)\ge f(9,4,3)&=&5(n-k-5)+(n-k-1)-9(2k-10)=6n-24k+64\geq 16.
\end{eqnarray*}
This completes the proof.\qed
\end{proof}

\begin{lemma}\label{lem_prep2}
$S_1>T_1+s(k-i+1)$ if $(s,i,t)$ is not one of $(7,5,3)$, $(7,4,3)$ or $(6,4,3)$.
\end{lemma}
\begin{proof}
Set $g(s,i,t)=S_1-T_1-s(k-i+1)$, so that  $g(s,i,t)=(s-i)(n+i-k-s)-s(k-i+1)$ by \eqref{H1234}. We have
\begin{eqnarray*}
\frac{\partial g}{\partial s}(s,i,t)&=&n+3i-2s-2k-1\geq  n+3i-2(k+i-t)-2k-1 \\ &=& n+i+2t-4k-1\geq (t+1)(k-t+1)+3t-4k\\
&=&(t-3)(k-t)+1>0
\end{eqnarray*}
for $t\ge 3$. Here, we have used the condition  $s\le k+i-t$ for the first inequality, and the conditions $n\ge (t+1)(k-t+1)$ and $i\ge t+1$ for the second. It follows that $g(s,i,t)$ is an increasing function in $s$.

Firstly, suppose that $t\ge 4$.
Recall that $s\ge t+3$ and $s\ge 2i-t$.
For $i=t+1$, we have \[
   g(s,t+1,t)\geq g(t+3,t+1,t)=2(n-k-2)-(t+3)(k-t)\geq (t-3)(k-t)-2\geq 0.
\]
For $i\geq t+2$, we have
\[
   g(s,i,t)\geq g(2i-t,i,t)\geq g(t+4,t+2,t)=2(n-k-2)-(t+4)(k-t-1)\geq (t-4)(k-t)+6 > 0,
\]
where the second inequality follows from that $g(2i-t,i,t)=i^2 + (n - 3k + t - 2)i + t + 2kt - nt - t^2$ is increasing in $i$ on $[t+2,\infty)$.

Now suppose that $t=3$ in the sequel. Then $i\ge t+1=4$, and $n\ge 4(k-2)$.
We deduce that $s\ge 2i-3$ from $i\le\frac{s+t}{2}$, and so
\[
  g(s,i,3)\geq g(2i-3,i,3)=i^2+i(n+1-3k)-3(n-2k+2).
\]
As a quadratic equation in $i$, the last expression is an increasing function in the interval $[5,\infty)$  since $\frac{-n+3k-1}{2}\le\frac{7-k}{2}<5$. Therefore,
\begin{eqnarray*}g(s,i,3)&\geq&g(2i-3,i,3)\geq g(9,6,3)=6^2+6(n-3k+1)-3(n-2k+2)\\
&=&36+3n-12k\geq 12 \end{eqnarray*}
for $i\geq 6$ as $n\geq 4(k-2)$.
 If $i\in\{4,5\}$, then we deduce that $s\geq 8$ as $s\geq\max\{t+3, 2i-3\}$ and   $\{(7,4,3), (6,4,3),(7,5,3)\}$ are excluded from consideration.  Moreover, it follows from $s+t-i\leq k$ that $k\geq 7$ if $s\geq 8$, $t=3$ and $i=4$.
  Correspondingly,  we have
$$g(s,4,3)\geq g(8,4,3)=4n-12k+8\geq 4k-24>0$$ and $$g(s,5,3)\geq g(8,5,3)=3n-11k+23\geq k-1>0.$$
This completes the proof.\qed
\end{proof}

\begin{lemma}\label{lem_prep3}
$S_2>T_1+s(k+i-s-t+1)-(S_2-T_2)$.
\end{lemma}
\begin{proof}
Set $h(s,i,t)=S_2-\left(T_1+s(k+i-s-t+1)-(S_2-T_2)\right)$, so that by \eqref{H1234} we have
\begin{align*}
h(s,i,t)&=(s+t-2i)(n+t+i-2k)+(i-t)(2n+t-2k-s)-s(k+i-s-t+1)\\
&=s^2+s(n + 3t -3k - i - 1) + i(2k - 2i) - tn.
\end{align*}
We have
\begin{eqnarray*}
\frac{\partial h}{\partial s}(s,i,t)
&=& n+3t+2s-3k-i-1\geq (t+1)(k-t+1)-3(k-t+1)+2s+2-i \\
&=& (t-2)(k-t+1)+2s+2-i> 0.
\end{eqnarray*}
It follows that $h(s,i,t)$ is an increasing function in $s$. Recall that we have $s\ge 2i-t$. Then  $$h(s,i,t)\geq h(2i-t,i,t)=(i-t)(2n+2t-2k-2i)-(2i-t)(k-i+1).$$
However,
 \begin{eqnarray*}
\frac{\partial h}{\partial i}(2i-t,i,t)&=&2n-4k+3t-2\ge 2(t+1)(k-t+1)-4k+3t-2 \\
&=&2(t-1)(k-t+1)-t+2\ge 6(t-1)-t+2=5t-4>0
\end{eqnarray*}
since $k\geq t+2$.
Therefore, $h(2i-t,i,t)$ is an increasing function in $i$, and we have  $$h(s,i,t)\geq h(t+2,t+1,t)=2(n-k-1)-(t+2)(k-t)\geq (t-2)(k-t)>0.$$
This completes the proof.\qed
\end{proof}

\begin{lemma}\label{lem_prep4}
$T_2+s(k+i-s-t+1) \le s(n-s+1)$.
\end{lemma}
\begin{proof}
Recall that $n\ge (t+1)(k-t+1)$ and $i\ge t+1$.
We have $$s(n-s+1)-T_2 -s(k+i-s-t+1)=(i-t)(n-2k+s+2t-2i)-s\geq (i-t)[(t+1)(k-t+1)-2k+s+2t-2i]-s.$$
Set $\varphi(k,s,i,t)=(i-t)((t+1)(k-t+1)-2k+s+2t-2i)-s$.
Note that  $\frac{\partial^2 \varphi}{\partial i^2}(k,s,i,t)=-4$, it follows that $\varphi(k,s,i,t)$ is a concave function in $i$ in $(-\infty,+\infty)$. Hence, $$\varphi(k,s,i,t)\geq \min\{\varphi(k,s,t+1,t),\varphi(k,s,(s+t)/2,t)\}$$ for $t+1\leq i\leq \frac{s+t}{2}$.
To complete the proof,  it suffices  to verify that
$\varphi(k,s,t+1,t)\geq 0$ and $\varphi(k,s,(s+t)/2,t) \geq 0$.
However,
\begin{eqnarray*}
\varphi(k,s,t+1,t)&=& (t+1)(k-t+1)-2k-2=-t^2+ kt - k - 1 \\
&\ge&-t^2 + (t+2)t - (t+2) - 1=t-3\geq 0,
\end{eqnarray*}
where the first inequality follows from $k\ge t+2$;
\begin{eqnarray*}
\varphi(k,s,(s+t)/2,t)&=&(st-s - t^2+t)k + st - s - st^2+ t^3 - t^2 - t  \\
&\ge& ((2t - 3)s - 2t^2 + t)/2
\ge (4t-9)/2>0,
\end{eqnarray*}
where the inequalities follow from $k\geq t+2$, $s\ge t+3$ and $t\ge 3$.
 This completes the proof.\qed
\end{proof}

We are now ready to prove Theorem \ref{lemmakeyineq}.

\noindent\textbf{Proof of Theorem \ref{lemmakeyineq}.} The right-hand side of \eqref{ff1} is equal to $\frac{(k-i+1)(k+i+1-s-t)}{(n+i-k-s)(n+t-k-i)}$ upon simplification, so \eqref{ff1} reduces to
\begin{eqnarray}\label{ff11}
  \frac{(n+i-k-s)(n+t-k-i)S_1 S_2}{(n-s+1)^2 T_1 T_2}>1.
\end{eqnarray}
We have $i\le\frac{s+t}{2}\le k$ by the assumption $s\le 2k-t$. By \eqref{H1234} we have $S_1=s(n-s+1)-i(k-i)< s(n-s+1)$, i.e., $\frac{S_1}{n-s+1}<s=\frac {s(k-i+1)}{k-i+1}$. By Lemma \ref{basefact}, we deduce that $(n-k-s+i)S_1> ((n-k-s+i)+(k-i+1))(S_1-s(k-i+1))$, i.e.,
\begin{eqnarray}\label{equa1}
   (n-k-s+i)S_1>(n-s+1)(S_1-s(k-i+1)).
\end{eqnarray}
Similarly, we have $S_2=s(n-s+1)-(s+t-i)(k+i-s-t)< s(n-s+1)$ and then by Lemma \ref{basefact} we obtain
$(n-k-i+t)S_2>(n-s+1)\left(S_2-s(k+i-s-t+1)\right)$. We have $S_1-s(k-i+1)>T_1$ by Lemma \ref{lem_prep2}.
If $S_2-s(k+i-s-t+1)\ge T_2$, then
 \begin{small}\begin{eqnarray*}
(n-k-s+i)S_1\cdot(n-k-i+t)S_2&>&(n-s+1)\left(S_1-s(k-i+1)\right)(n-s+1)(S_2-s(k+i-s-t+1)) \\
&>&(n-s+1)^2T_1T_2,
\end{eqnarray*}\end{small}
and the desired result follows. It remains to consider the case
\begin{eqnarray}\label{equa3}
S_2-T_2<s(k+i-s-t+1).
 \end{eqnarray}
First suppose that $(s,i,t)$ is not one of $(8,6,5)$, $(7,5,4)$, $(8,5,3)$, $(8,4,3)$, $(7,5,3)$, $(7,4,3)$ and $(6,4,3)$.
By Lemma \ref{lem_prep1} we have
\begin{eqnarray}\label{equac2}
  S_1-s(k-i+1)\geq T_1+s(k+i-s-t+1)-(S_2-T_2).
\end{eqnarray}
By Lemma \ref{lem_prep3} we have $S_2-\left(T_1+s(k+i-s-t+1)-(S_2-T_2)\right)>0$. Together with \eqref{equa3}, we deduce that
\begin{eqnarray}\label{st}
T_1+s(k+i-s-t+1)-(S_2-T_2)<S_2<T_2 + s(k+i-s-t+1).
\end{eqnarray}
By (\ref{equa3}) and (\ref{st}), we have $S_2 > T_1$, so
\begin{eqnarray}\nonumber
[T_1+s(k+i-s-t+1)-(S_2-T_2)]S_2&>&T_1[S_2 + s(k+i-s-t+1)-(S_2-T_2)]\\ \label{equac1}&=&T_1[T_2+s(k+i-s-t+1)].
 \end{eqnarray}
By Lemma \ref{lem_prep4}, we have $\frac{T_2+s(k+i-s-t+1)}{n-s+1} \le s$. By Lemma \ref{basefact}, this can be reformulated as
\begin{eqnarray}\label{equac3}
  (n-k-i+t)[T_2+s(k+i-s-t+1)]\geq(n-s+1)T_2.
\end{eqnarray}
We then have
\begin{eqnarray*}
&&(n-k-s+i)(n-k-i+t)S_1S_2   \\
&\overset{\text{(\ref{equa1})}}{>}& (n-k-i+t)(n-s+1)[S_1-s(k-i+1)]S_2  \\
&\overset{\text{(\ref{equac2})}}{\ge}& (n-k-i+t)(n-s+1)\left(T_1+s(k+i-s-t+1)-(S_2-T_2)\right)S_2 \\
&\overset{\text{(\ref{equac1})}}{>}& (n-k-i+t)(n-s+1)T_1[T_2+s(k+i-s-t+1)]  \\
&\overset{\text{(\ref{equac3})}}{\ge}& (n-s+1)^2T_1T_2.
\end{eqnarray*}
It remains to consider the cases where $(s,i,t)$ is one of $(8,6,5)$, $(7,5,4)$, $(8,5,3)$, $(8,4,3)$, $(7,5,3)$, $(7,4,3)$ and $(6,4,3)$. For each tuple $(n,k,s,i,t)$, we define
\[
  T(n,k,s,i,t)=\frac{(n-k-s+i)(n-k-i+t)S_1 S_2}{(n-s+1)^2 T_1 T_2}.
\]
We only give details for the triple $(8,6,5)$ here, since the other triples are handled similarly (the proofs for other cases are in the appendix).
Suppose that $(s,i,t)=(8,6,5)$. We have  $k\geq s+t-i=7$ and $n\ge 6(k-4)$.
Since $T(18,7,8,6,5)=615/572>1$, we can assume that $k\ge 8$ and $n\ge 6(k-4)$, or $k=7$ and $n\ge 19$.
We have
\[
  (8n-6k-20)(8n-7k-7)-(8n-4.5k-20)(8n-8.5k-7)=3k(5k - 26)/4>0.
\]
It follows that
 \begin{eqnarray*}
 T(n,k,8,6,5)&=&\frac{(n-k-2)(n-k-1)(8n-6k-20)(8n-7k-7)}
{(n-7)^2(6n-4k-16)(7n-6k-6)}\\
&\geq&\frac{(n-k-2)(n-k-1)(8n-4.5k-20)(8n-8.5k-7)}
{(n-7)^2(6n-4k-16)(7n-6k-6)}.
\end{eqnarray*}
On the one hand,
\begin{eqnarray*}
(n-k-2)(8n-8.5k-7)-(n-7)(6n-4k-16)&=&(4n^2  - 25kn  +  70n +17k^2- 8k- 196)/2\\
&\geq&(11k^2 - 140k + 428)/2 >0
\end{eqnarray*}
for $k\ge 8$ and $n\ge 6(k-4)$,
and $(4n^2  - 25kn  +  70n +17k^2- 8k- 196)/2\ge 15>0$ for $k=7$ and $n\ge 19$.
On the other hand,
\begin{eqnarray*}(n-k-1)(8n-4.5k-20)-(n-7)(7n-6k-6)&=&( 2n^2  - 13kn + 54n+9k^2 - 35k - 44)/2\\ &\geq& (3k^2 + 25k - 188)/2>0.\end{eqnarray*}
We thus conclude that $T(n,k,8,6,5)>1$ as desired. This completes the proof of Theorem \ref{lemmakeyineq}.\qed

\section{Proof of Theorem \ref{maintheo}}\label{section4}
Suppose $n\geq (t+1)(k-t+1)$ and $t\geq 3$.
Let $\mathcal A$ and $\mathcal B$ be two left-compressed cross-$t$-intersecting families of $\binom{[n]}{k}$ such that $|\mathcal A||\mathcal B|$ is maximum,
$g(\mathcal A)\in G_*(\mathcal A)$ and $g(\mathcal B)\in G_*(\mathcal B)$  the generating sets of $\mathcal A$ and $\mathcal B$, respectively.
Clearly, \[\mbox{$|\mathcal A||\mathcal B|\geq \binom{n-t}{k-t}^2\geq |\mathcal F(n,k,t,1)|^2.$}\]
Set $s=\max\{s^+(g(\mathcal A)),s^+(g(\mathcal B))\}$.
If $s=t$, as  $g(\mathcal A)$ and $g(\mathcal B)$ are cross-$t$-intersecting, we immediately obtain that $g(\mathcal A)=g(\mathcal B)=\{[t]\}$, hence \[\mbox{$\mathcal A=\mathcal B=\left\{A\in\binom{[n]}{k}: [t]\subset A \right\}$ and $|\mathcal A||\mathcal B|=\binom{n-t}{k-t}^2.$}\]
So we suppose $s\geq t+1$ in the sequel.  Let $i$ be the minimum integer such that $g^*_i(\mathcal A)\neq\emptyset$.   Then $g^*_{s+t-i}(\mathcal B)\neq\emptyset$ by (ii) of Lemma \ref{lemmag1}. By symmetry, we may assume $i\leq (s+t)/2$, that is, $s\geq 2i-t$. If $i=t$, then $g(\mathcal A)=\binom{[s]}t$ and $g(\mathcal B)=\{[s]\}$ holds by the assumption that $|\mathcal A||\mathcal B|$ is maximum.
Hence\[\mbox{$|\mathcal A|=\sum_{t\leq i\leq s}\binom{s}{i}\binom{n-s}{k-i}$ and $|\mathcal B|=\binom{n-s}{k-s}$.}\]
Set $\mathcal A_1=\{A\in \binom{[n]}{k}: |A\cap [s-1]|\geq t\}$ and $\mathcal B_1=\{B\in\binom{[n]}{k}: [s-1]\subset B\}$.
Clearly, $\mathcal A_1$ and $\mathcal B_1$ are cross-$t$-intersecting families with
\[\mbox{$|\mathcal A_1|=\sum_{t\leq i\leq s-1}\binom{s-1}{i}\binom{n-s+1}{k-i}$ and $|\mathcal B_1|=\binom{n-s+1}{k-s+1}$.}\]
Then,  $|\mathcal A|-|\mathcal A_1|=\binom{s-1}{t-1}\binom{n-s}{k-t}$ and $|\mathcal B_1|-|\mathcal B|=\binom{n-s}{k-s+1}$.   It follows from the assumption  $n\geq (t+1)(k-t+1)$ that  $$s(n-k)-t(n-s+1)=(s-t)n-s(k-t)-t>t(s-t-1)(k-t+1)>0,$$ and so we obtain
\begin{eqnarray*}\frac{\binom{s-1}{t-1}\binom{n-s+1}{k-s+1}}{\binom{s}{t}\binom{n-s}{k-s+1}}=\frac{t(n-s+1)}{s(n-k)}<1,\end{eqnarray*}
that is, \[\binom{s-1}{t-1}\binom{n-s+1}{k-s+1}<\binom{s}{t}\binom{n-s}{k-s+1}.\]
Therefore,
  \begin{eqnarray*}|\mathcal A_1||\mathcal B_1|-|\mathcal A||\mathcal B|&=&\binom{n-s}{k-s+1}\sum_{t\leq i\leq s}\binom{s}{i}\binom{n-s}{k-i}-\binom{s-1}{t-1}\binom{n-s}{k-t}\binom{n-s+1}{k-s+1}\\&>&\binom{n-s}{k-t}\left(\binom{s}{t}\binom{n-s}{k-s+1}-\binom{s-1}{t-1}
 \binom{n-s+1}{k-s+1}\right)>0,\end{eqnarray*}
 contradicting the maximality of $|\mathcal A||\mathcal B|$. So, in the following, we assume $i\geq t+1$ and $s\geq 2i-t\geq t+2$. We will distinguish three cases to complete the proof.
\medskip
\newline
\textbf{Case 1}: $s=t+2$. Then, $i=t+1$, $g^*_{t+1}(\mathcal A)\neq\emptyset$ and $g^*_{t+1}(\mathcal B)\neq\emptyset$. Suppose $[t]\in g(\mathcal A)\cup g(\mathcal B)$. Furthermore, we may assume $[t]\in g(\mathcal A)$. If $[t]\in g(\mathcal B)$, then $g(\mathcal A)=g(\mathcal B)=\{[t]\}$, which implies $s=s^+(g(\mathcal A))=s^+(g(\mathcal B))=t$, yielding a contradiction. Therefore, $[t]\not\in g(\mathcal B)$, and it follows that \[\mbox{$g(\mathcal A)=\{[t]\}\cup\{ [t+2]\setminus \{i\}: 1\leq i\leq t\}$ and $g(\mathcal B)=\{[t+1],[t+2]\setminus\{t+1\}\}$}.\]
Hence, \[\mbox{$\mathcal A=\{A\in\binom{[n]}{k}: \mbox{$[t]\subset A$ or $|A\cap [t+2]|\geq t+1$}\}$ and $\mathcal B=\{B\in\binom{[n]}{k}: \mbox{$[t]\subset B$ and $|B\cap [t+2]|\geq t+1$}\}$}.\]
 Note that $t\binom{n-t-2}{k-t-1}\leq\binom{n-t-2}{k-t}$ when $n\geq (t+1)(k-t+1)$. It follows
\[|\mathcal A||\mathcal B|=\left(\binom{n-t}{k-t}+t\binom{n-t-2}{k-t-1}\right)\left(\binom{n-t}{k-t}-\binom{n-t-2}{k-t}\right)<\binom{n-t}{k-t}^2,\]
yielding a contradiction. Therefore, $[t]\not\in g(\mathcal A)\cup g(\mathcal B)$, which implies $|F|\geq t+1$ for all  $F\in g(\mathcal A)\cup g(\mathcal B)$. It follows that  $g(\mathcal A)=g(\mathcal B)=\binom{[t+2]}{t+1}$.  Hence $\mathcal A=\mathcal B=\mathcal F(n,k,t,1)=\{A\in \binom{[n]}{k}: |A\cap [t+2]|\geq t+1\}$, and the maximality of $|\mathcal A||\mathcal B|$ implies $\mathcal F(n,k,t,1)$ is a  maximum $t$-intersecting subfamily of $\binom{[n]}k$, then $n=(t+1)(k-t+1)$.
\medskip
\newline
\textbf{Case 2}: $s\geq t+3$ and $(s,i,t)\neq (6,4,3)$.  Then, $g^*_{s+t-i}(\mathcal B)\neq\emptyset$ holds by (ii) of Lemma (\ref{lemmag1}). Set $\mathcal A_1=\mathcal A\cup \mathcal D(g^{*}_{i}(\mathcal A)')$ and $\mathcal B_1=\mathcal B\setminus \mathcal D(g^{*}_{s+t-i}(\mathcal B))$.  Then,  $\mathcal A_1$ and $\mathcal B_1$ are also cross-$t$-intersecting families of $\binom{[n]}{k}$, by (iii) of Lemma (\ref{lemmag1})  and the maximality of $|\mathcal A||\mathcal B|$, we have
\begin{eqnarray*}
|\mathcal A_1||\mathcal B_1|=\left(|\mathcal A|+|g^*_i(\mathcal A)|\binom{n-s}{k-i+1}\right)\left(|\mathcal B|-|g^*_{s+t-i}(\mathcal B)|\binom{n-s}{k-s-t+i}\right)\leq|\mathcal A||\mathcal B|.
\end{eqnarray*}
By Lemma \ref{basefact}, we obtain
\begin{eqnarray}
\label{ineq31}
\frac{|\mathcal B|}{|\mathcal A|+|g^{*'}_{i}(\mathcal A)|\binom {n-s}{k-i+1}}
\leq \frac{|g^{*}_{s+t-i}(\mathcal B)|\binom {n-s}{k-s-t+i}}{|g^{*'}_{i}(\mathcal A)|\binom {n-s}{k-i+1}}.
\end{eqnarray}
 Set $\mathcal A_2=\mathcal A\setminus\mathcal D(g^{*}_{i}(\mathcal A))$ and $\mathcal B_2=\mathcal B\cup \mathcal D(g^{*}_{s+t-i}(\mathcal B)')$.   Similarly,  $\mathcal A_2$ and $\mathcal B_2$ are also cross-$t$-intersecting families with
\begin{eqnarray*}
|\mathcal A_2||\mathcal B_2|=\left(|\mathcal A|-|g^*_i(\mathcal A)|\binom{n-s}{k-i}\right)\left(|\mathcal B|+|g^*_{s+t-i}(\mathcal B)|\binom{n-s}{k-s-t+i+1}\right)\leq|\mathcal A||\mathcal B|,
\end{eqnarray*}
and so
\begin{eqnarray}
\label{ineq32}
\frac{|\mathcal A|}{|\mathcal B|+|g^{*}_{s+t-i}(\mathcal B)|\binom {n-s}{k-s-t+i+1}}
\leq \frac{|g^*_i(\mathcal A)|\binom {n-s}{k-i}}{|g^*_{s+t-i}(\mathcal B)|\binom {n-s}{k-s-t+i+1}}
\end{eqnarray}
holds by Lemma \ref{basefact}. Combining with (\ref{ineq31}) and (\ref{ineq32}), we have
\begin{eqnarray}
\label{ineq33}
\frac{|\mathcal A|} {|\mathcal A|+|g^*_i(\mathcal A)|\binom {n-s}{k-i+1}}
\frac{|\mathcal B|}{|\mathcal B|+|g^*_{s+t-i}(\mathcal B)|\binom {n-s}{k-s-t+i}}
\leq \frac{\binom {n-s}{k-i} \binom {n-s}{k-s-t+i}}{\binom {n-s}{k-i+1}\binom {n-s}{k-s-t+i+1}}.
\end{eqnarray}
Set $\nabla(g_i^*(\mathcal A)')=\{F\in\binom{[s-1]}{i}: \mbox{$E\subset F$ for some $E\in g_i^*(\mathcal A)'$}\}$. Then $\mathcal D(\nabla(g_i^*(\mathcal A)'))\subseteq \mathcal A$ and
\[
\frac{|\nabla(g_i^*(\mathcal A)')|}{|g_i^*(\mathcal A)|}=\frac{|\nabla(g_i^*(\mathcal A)')|}{|g_i^*(\mathcal A)'|}\geq \frac{\binom{s-1}{i}}{\binom{s-1}{i-1}}
\]
 holds by Theorem \ref{nmm}, and so we obtain

\begin{eqnarray*}|\mathcal A| \geq \mathcal D(\nabla(g_i^*(\mathcal A)')\cup g^*_i(\mathcal A))\geq \frac{|g^*_i(\mathcal A)|}{\binom{s-1}{i-1}}\left(\binom{s-1}{i-1}\binom{n-s}{k-i}+\binom{s-1}{i}\binom{n-s+1}{k-i}\right). \end{eqnarray*}
Hence
\begin{eqnarray*}\frac{|\mathcal A|}{|\mathcal A|+|g^*_i(\mathcal A)|\binom{n-s}{k-i+1}}\geq\frac{\binom{s-1}{i-1}\binom{n-s}{k-i}+\binom{s-1}{i}\binom{n-s+1}{k-i}}
{\binom{s-1}{i-1}\binom{n-s+1}{k-i+1}+\binom{s-1}{i}\binom{n-s+1}{k-i}}=\frac{(k-i+1)S_1}{(n-s+1)T_1}.
 \end{eqnarray*}
Similarly,
\begin{eqnarray*}\frac{|\mathcal B|}{|\mathcal B|+|g^*_{s+t-i}(\mathcal B)|\binom{n-s}{k+i+1-s-t}}\geq\frac{\binom{s-1}{s+t-i-1}\binom{n-s}{k+i-s-t}+\binom{s-1}{s+t-i}\binom{n-s+1}{k+i-s-t}}
{\binom{s-1}{s+t-i-1}\binom{n-s+1}{k+i+1-s-t}+\binom{s-1}{s+t-i}\binom{n-s+1}{k+i-s-t}}=\frac{(k+i+1-s-t)S_2}{(n-s+1)T_2}.
 \end{eqnarray*}
Consequently, we have
\begin{small}
\[\frac{(k-i+1)(k+i+1-s-t)S_1S_2}{(n-s+1)^2T_1T_2}\leq \frac{|\mathcal A|} {|\mathcal A|+|g^*_i(\mathcal A)|\binom {n-s}{k-i+1}}
\frac{|\mathcal B|}{|\mathcal B|+|g^*_{s+t-i}(\mathcal B)|\binom {n-s}{k-s-t+i}}\leq \frac{\binom {n-s}{k-i} \binom {n-s}{k-s-t+i}}{\binom {n-s}{k-i+1}\binom {n-s}{k-s-t+i+1}}.\]\end{small}
 However, as $(s,i,t)\neq (6,4,3)$, by Theorem \ref{lemmakeyineq}, we have  \[\frac{(k-i+1)(k+i+1-s-t)S_1S_2}{(n-s+1)^2T_1T_2}>\frac{\binom{n-s}{k-i}\binom {n-s}{k-s-t+i}}{\binom {n-s}{k-i+1}\binom {n-s}{k-s-t+i+1}},\]
 yielding a contradiction.
\medskip
\newline
  \textbf{Case 3:} $(s,i,t)=(6,4,3)$. Then $g_4^*(\mathcal A)\neq\emptyset$ and $g^*_5(\mathcal B)\neq\emptyset$.  We distinguish three cases to complete the proof. For simplicity, we denote $\{i_1,i_2,...,i_{m}\}$ as $i_1i_2...i_{m}$ in the following.
\medskip
\newline
\textbf{ Subcase 3.1:} $|g_3(\mathcal A)|\geq 2$. By (i) of Lemma \ref{lemmag1},  we have  $\{123, 124\}\subset g(\mathcal A)$. Note that
$g(\mathcal A)$ and $g(\mathcal B)$ are cross-$3$-intersecting, we have $[4]\subset F$ for all $F\in g(\mathcal B)$. It follows from $g^*_5(\mathcal B)\neq \emptyset$ that $g(\mathcal B)=\{12345, 12346\}$. Moreover, by the maximality of $|\mathcal A||\mathcal B|$, we get that  $$g(\mathcal A)=\{123,124,134,234,1256,1356,1456,2356,2456,3456\}.$$
Then, $|\mathcal A|=4\binom{n-4}{k-3}+\binom{n-4}{k-4}+6\binom{n-6}{k-4}$ and $|\mathcal B|=2\binom{n-6}{k-5}+\binom{n-6}{k-6}$. Set $\mathcal A_1=\{A\in\binom{[n]}{k}: |A\cap [4]|\geq 3\}$ and $\mathcal B_1=\{B\in\binom{[n]}{k}: [4]\subset B\}$, it is clear that $\mathcal A_1$ and $\mathcal B_1$ are cross-3-intersecting families with $|\mathcal A_1|=\binom{n-4}{k-4}+4\binom{n-4}{k-3}$ and $|\mathcal B_1|=\binom{n-4}{k-4}$. Then
\begin{eqnarray*}|\mathcal A_1||\mathcal B_1|-|\mathcal A||\mathcal B|=|\mathcal A_1|\left(|\mathcal B|+\binom{n-6}{k-4}\right)-\left(|\mathcal A_1|+6\binom{n-6}{k-4}\right)|\mathcal B|=(|\mathcal A_1|-6|\mathcal B|)\binom{n-6}{k-4}.\end{eqnarray*}
 Note that \begin{eqnarray}\label{eq2}\frac{\binom{m}{j}}{\binom{m}{j-1}}=\frac{m-j+1}{j}>3\end{eqnarray}
 if $m\geq 4j$.
  By the assumption $n\geq 4(k-2)$, $|\mathcal A_1|=\binom{n-4}{k-4}+4\binom{n-4}{k-3}>13\binom{n-4}{k-4}>13|\mathcal B|$. Therefore, $|\mathcal A_1||\mathcal B_1|>|\mathcal A||\mathcal B|$,
contradicting the maximality of $|\mathcal A||\mathcal B|$.
\medskip
\newline
\textbf{Subcase 3.2:} $|g_3(\mathcal A)|=1.$ Then $g_3(\mathcal A)=\{[3]\}$, which implies $[3]\subset F$ for all $F\in g(\mathcal B)$. Moreover, $1246\in g^*_4(\mathcal A)$ and $12356\in g^*_5(\mathcal B)$ since $g^*_4(\mathcal A)\neq\emptyset$ and $124\not\in g(\mathcal A)$. Then $123\not\in g(\mathcal B)$. Therefore, $g_3(\mathcal B)=\emptyset$ and   either $g_4(\mathcal B)=\{[4]\}$ or $\emptyset$.

If $g_4(\mathcal B)=\emptyset$, since $124\not\in g(\mathcal A)$ and $123\in g(\mathcal A)$,  it follows that $g(\mathcal B)=\{12346, 12345, 12356\}$. By the maximality of $|\mathcal A||\mathcal B|$, we have  $$g(\mathcal A)=\{123,1245,1246,1256,1345,1346,1356,1456, 2345,2346,2356,2456,3456\}.$$ Hence,
\[\mbox{$|\mathcal A|=\binom{n-6}{k-6}+6\binom{n-6}{k-5}+15\binom{n-6}{k-4}+\binom{n-6}{k-3}$ and $|\mathcal B|=\binom{n-6}{k-6}+3\binom{n-6}{k-5}$.}\]
  Set $\mathcal A_1=\{A\in\binom{[n]}{k}: \mbox{$|A\cap [5]|\geq 4$ or $[3]\subset A$}\}$ and $\mathcal B_1=\{B\in\binom{[n]}{k}: \mbox{$|B\cap [5]|\geq 4$ and $[3]\subset B$}\}$. Clearly, $\mathcal A_1$ and $\mathcal B_1$ are cross-$3$-intersecting with
  \[\mbox{$|\mathcal A_1|=\binom{n-5}{k-5}+5\binom{n-5}{k-4}+\binom{n-5}{k-3}=|\mathcal A|-9\binom{n-6}{k-4}$ and $|\mathcal B_1|=2\binom{n-5}{k-4}+\binom{n-5}{k-5}=|\mathcal B|+2\binom{n-6}{k-4}$.}\]
  By the assumption $n\geq 4(k-2)$ and (\ref{eq2}),  $\binom{n-6}{k-4}>\binom{n-6}{k-5}>\binom{n-6}{k-6}$ and $\binom{n-6}{k-4}>3\binom{n-6}{k-5}$.
Therefore,
\begin{eqnarray*}
|\mathcal A_1||\mathcal B_1|-|\mathcal A||\mathcal B|&=&2\binom{n-6}{k-4}\left(\binom{n-6}{k-6}+6\binom{n-6}{k-5}+15\binom{n-6}{k-4}+\binom{n-6}{k-3}\right)\\& &-9\binom{n-6}{k-4}\left(\binom{n-6}{k-6}+3\binom{n-6}{k-5}\right)-18\binom{n-6}{k-4}^2\\&=&\binom{n-6}{k-4}\left(-7\binom{n-6}{k-6}-15\binom{n-6}{k-5}+12\binom{n-6}{k-4}+2\binom{n-6}{k-3}\right)>0,
\end{eqnarray*}
 contradicting the maximality of $|\mathcal A||\mathcal B|$.

 If $g_4(\mathcal B)=\{[4]\}$, since $123\in g(\mathcal A)$ and $12356\in g^*_5(\mathcal B)$,  we get that
 $g(\mathcal B)=\{1234,12356\}.$ By the maximality of $|\mathcal A||\mathcal B|$, we have
$g(\mathcal A)=\{123,1245, 1246, 1345, 1346, 2345,2346\}$.
 Hence,
 \[|\mathcal A|=\binom{n-6}{k-3}+9\binom{n-6}{k-4}+6\binom{n-6}{k-5}+\binom{n-6}{k-6} \  {\rm and} \  |\mathcal B|=\binom{n-6}{k-4}+3\binom{n-6}{k-5}+\binom{n-6}{k-6}.\]
 Set $\mathcal A_1=\mathcal A\setminus\mathcal D(g^*_4(\mathcal A))$ and $\mathcal B_1=\mathcal B\cup \mathcal D(g^*_5(\mathcal B)')$, where $g^*_4(\mathcal A)=\{1246,1346,2346\}$ and $g^*_5(\mathcal B)=\{12356\}$. Clearly, $\mathcal A_1$ and $\mathcal B_1$ are cross-3-intersecting with
 \[|\mathcal A_1|=|\mathcal A|-3\binom{n-6}{k-4} \ {\rm and}\  |\mathcal B_1|=|\mathcal B|+\binom{n-6}{k-4}.\]
 By the assumption $n\geq 4(k-2)$ and (\ref{eq2}),  we also have
 \begin{eqnarray*}
 |\mathcal A_1||\mathcal B_1|-|\mathcal A||\mathcal B|=\binom{n-6}{k-4}\left(\binom{n-6}{k-3}+3\binom{n-6}{k-4}-3\binom{n-6}{k-5}-2\binom{n-6}{k-6}\right)>0,
 \end{eqnarray*}
 contradicting the maximality of $|\mathcal A||\mathcal B|$.

 \medskip
\noindent
\textbf{Subcase 3.3:} $g_3(\mathcal A)=\emptyset$. Then  $1236\in g^*_4(\mathcal A)$ since $g^*_4(\mathcal A)\neq\emptyset$. If $g_3(\mathcal B)\neq \emptyset$,  then  $123\in g_3(\mathcal B)$, which implies $123\subseteq F$ for each $F\in g(\mathcal A)$. It follows  that
$g(\mathcal A)=\{1234,1235,1236\}$ and $g(\mathcal B)=\{123, 12456, 13456,23456\}$. By the assumption $n\geq 4(k-2)$ and (\ref{eq2}), $\binom{n-6}{k-3}>3\binom{n-6}{k-4}>9\binom{n-6}{k-5}$, and then
 \[|\mathcal A||\mathcal B|=\left(\binom{n-3}{k-3}-\binom{n-6}{k-3}\right)\left(\binom{n-3}{k-3}+3\binom{n-6}{k-5}\right)<\binom{n-3}{k-3}^2,\]
  yielding a contradiction. Now, we suppose $g_3(\mathcal B)=\emptyset$. Then $g_3(\mathcal A)=g_3(\mathcal B)=\emptyset$, which implies $|C\cap [6]|\geq 4$ for each $C\in\mathcal A\cup \mathcal B$. Hence $\mathcal D(\binom{[6]}{5})\subset \mathcal A\cap \mathcal B$ by the maximality of $\mathcal A$ and $\mathcal B$.
 If $g_4(\mathcal B)=\emptyset$, we immediately have  $\mathcal A=\{A\in \binom{[n]}{k}: |A\cap[6]|\geq 4\}$ and $\mathcal B=\{B\in \binom{[n]}{k}: |B\cap[6]|\geq 5\}$.
 Similarly,
\begin{small}
\begin{eqnarray*}
 |\mathcal A||\mathcal B|-|\mathcal F(n,k,3,1)|^2&=& \left(|\mathcal F(n,k,3,1)|+ 10\binom {n-6}{k-4} \right) \left(|\mathcal F(n,k,3,1)|-5\binom {n-6}{k-4} \right)-|\mathcal F(n,k,3,1)|^2 \\
&=& 5\binom {n-6}{k-4}\left(5\binom {n-5}{k-4}+\binom {n-5}{k-5}-10\binom {n-6}{k-4}\right) \\
&=& 5\binom {n-6}{k-4}\binom {n-5}{k-5}\left(1-\frac {5(n-k)(n-2k+3)}{(k-4)(n-5)}\right)<0,
 \end{eqnarray*}
\end{small}
\newline
contradicting the maximality of $|\mathcal A||\mathcal B|$.
 Suppose $g_4(\mathcal B)\neq\emptyset$. Then, $g_4(\mathcal A)$ and $g_4(\mathcal B)$ are nonempty cross-$3$-intersecting families of $\binom{[6]}{4}$, and
\[\mbox{$|\mathcal A|=|g_4(\mathcal A)|\binom{n-6}{k-4}+6\binom{n-6}{k-5}+\binom{n-6}{k-6}$, $|\mathcal B|=|g_4(\mathcal B)|\binom{n-6}{k-4}+6\binom{n-6}{k-5}+\binom{n-6}{k-6}$}.\] By Theorem \ref{wzn}, $|g_4(\mathcal A)|+|g_4(\mathcal B)|\leq 10$, and so
\begin{eqnarray*}
|\mathcal A||\mathcal B|&=&\left(|g_4(\mathcal A)|\binom{n-6}{k-4}+6\binom{n-6}{k-5}+\binom{n-6}{k-6}\right)\left(|g_4(\mathcal B)|\binom{n-6}{k-4}+6\binom{n-6}{k-5}+\binom{n-6}{k-6}\right)\\ &\leq&
\left(5\binom{n-6}{k-4}+6\binom{n-6}{k-4}+\binom{n-6}{k-6}\right)\left(5\binom{n-6}{k-4}+6\binom{n-6}{k-5}+\binom{n-6}{k-6}\right)\\ &=&
\left(5\binom{n-5}{k-4}+\binom{n-5}{k-5}\right)^2=|\mathcal F(n,k,3,1)|^2,
\end{eqnarray*}
equality holds if and only if  $g_4(\mathcal A)=g_4(\mathcal B)=\binom{T}{4}$ for some $T\in\binom{[6]}{5}$. However, $g_3(\mathcal A)=g_3(\mathcal B)=\emptyset$, by (i) of Lemma   \ref{lemmag1}, we can deduce that   $g_4(\mathcal A)$ and $g_4(\mathcal B)$ are both left-compressed, hence $g_4(\mathcal A)=g_4(\mathcal B)=\binom{[5]}{4}$, which implies $g_4^*(\mathcal A)=g_4^*(\mathcal B)=\emptyset$. Therefore, either $|\mathcal A||\mathcal B|<|\mathcal F(n,k,3,1)|^2$ or $g_4^*(\mathcal A)=g_4^*(\mathcal B)=\emptyset$, contradicting the assumption.
This completes the proof under the assumption that $\mathcal A$ and $\mathcal B$ are left-compressed.
However, in this case of $\mathcal A=\mathcal B=\{A \in \binom{[n]}k: T\subseteq A\}$ for some $T\in \binom{[n]}t$, all $\Delta_{ij}(\mathcal{A})=\Delta_{ij}(\mathcal{A})$, $1\le i,j \le n$ and $i\neq j$, are isomorphic.
This means that the left-compression operation $\Delta_{ij}$ does not change the structure of a full $t$-star under isomorphism.
So the pair of extremal families are unique under isomorphism.
\qed

\section*{Acknowledgements}
The authors thank Tao Feng, Jian Wang and Jun Wang for many valuable suggestions and  the anonymous referees' valuable comments to improve our presentation
of the paper.
The first author is  supported by
the National Natural Science Foundation of China  (No.12371332 and
No.11971439);
The second author is supported by the National Natural Science Foundation of China (No. 11901193), the National Natural Science Foundation of Hunan Province, China (No. 2023JJ30385).

\section{Appendix}
Recall that \[
  T(n,k,s,i,t)=\frac{(n-k-s+i)(n-k-i+t)S_1 S_2}{(n-s+1)^2 T_1 T_2}.
\]
\begin{lemma}
For each $(s,i,t)\in \{(7,5,4),(8,5,3),(8,4,3),(7,5,3),(7,4,3)\}$, we have
$T(n,k,s,i,t)>1$.
\end{lemma}
\begin{proof}
{\bf Case 1}. $(s,i,t)=(8,5,3)$.
\begin{eqnarray*}
 T(n,k,8,5,3)&=&\frac{(n-k-3)(n-k-2)(8n-5k-31)(8n-6k-20)}
{(n-7)^2(5n-2k-22)(6n-4k-16)}.
\end{eqnarray*}
Since $n\geq (3+1)(k-3+1)=4k-8$ and $k\geq s+t-i=6$, we have
\begin{eqnarray*}
&&(n-k-2)(8n-5k-31)-(n-7)(6n-4k-16)\\&=&2n^2-(9k-11)n+(5k^2+13k-50)\\&\geq&2(4k-8)^2-(9k-11)(4k-8)+(5k^2+13k-50)\\&=&k^2+k-10>0
\end{eqnarray*}
and
\begin{eqnarray*}
&&(n-k-3)(8n-6k-20)-(n-7)(5n-2k-22)\\
&=&3n^2-(12k-13)n+(6k^2+24k-94)\\
&\geq&3(4k-8)^2-(12k-13)(4k-8)+(6k^2+24k-94)\\
&=&6k^2-20k-6>0.
\end{eqnarray*}

{\bf Case 2}.   $(s,i,t)=(8,4,3)$.
\begin{eqnarray*}
  T(n,k,8,4,3)&=&\frac{(n-k-4)(n-k-1)(8n-4k-40)(8n-7k-7)}{(n-7)^2(4n-24)(7n-6k-6)}\\&>&\frac{(n-k-4)(n-k-1)(8n-2k-40)(8n-9k-7)}{(n-7)^2(4n-24)(7n-6k-6)}.
\end{eqnarray*}
Since $n\geq (3+1)(k-3+1)=4k-8$ and $k\geq s+t-i=7$,
\begin{eqnarray*}
&&(n-k-1)(8n-2k-40)-(n-7)(7n-6k-6)\\
&=&n^2-(4k-7)n+(2k^2-2)\\&\geq&(4k-8)^2-(4k-7)(4k-8)+(2k^2-2)\\&=&2k^2-4k+6>0
\end{eqnarray*}
and
\begin{eqnarray*}
(n-k-4)(8n-9k-7)-(n-7)(4n-24)&=&4n^2 + (13 - 17k)n + 9k^2 + 43k - 140 \\
&\ge &5k^2 - 25k + 12>0.
\end{eqnarray*}

{\bf Case 3}. $(s,i,t)=(7,5,4)$.

\begin{eqnarray*}
  T(n,k,7,5,4)&=&\frac{(n-k-2)(n-k-1)(7n-5k-17)(7n-6k-6)}{(n-6)^2(5n-3k-13)(6n-5k-5)}\\&>&\frac{(n-k-2)(n-k-1)(7n-3.75k-17)(7n-7.25k-6)}{(n-6)^2(5n-3k-13)(6n-5k-5)}
  \\&>& \frac{(n-k-2)(n-k-1)(7n-3.75k-20)(7n-7.25k-4)}{(n-6)^2(5n-3k-13)(6n-5k-5)}.
\end{eqnarray*}
Since $n\geq (4+1)(k-4+1)=5k-15$, we have
\begin{eqnarray*}
&&(n-k-1)(7n-3.75k-20)-(n-6)(6n-5k-5)\\&=&n^2-(5.75k-14)n+(3.75k^2-6.25k-10)\\&\geq&(5k-15)^2-(5.75k-14)(5k-15)+(3.75k^2-6.25k-10)=5
\end{eqnarray*}
and
\begin{eqnarray*}
&&(n-k-2)(7n-7.25k-4)-(n-6)(5n-3k-13)\\&=&2n^2 -(11.25k-25)n+(7.25k^2+0.5k-70)\\&\ge&2(5k-15)^2-(11.25k-25)(5k-15)+(7.25k^2+0.5k-70)\\&=&k^2 - 5.75 + 5>0.
\end{eqnarray*}
Hence $ T(n,k,7,5,4)>1$.

{\bf Case 4}. $(s,i,t)=(7,5,3)$.
\begin{eqnarray*}
  T(n,k,7,5,3)&=&\frac{(n-k-2)^2(7n-5k-17)^2}{(n-6)^2(5n-3k-13)^2}.
\end{eqnarray*}
Since $n\geq 4k-8$ and $k\geq 5$, we have
\begin{eqnarray*}
&&(n-k-2)(7n-5k-17)-(n-6)(5n-3k-13)\\&=&2n^2-(9k-12)n+5k^2+9k-44\\&\geq&k^2+k-12>0.
\end{eqnarray*}
Hence $T(n,k,7,5,3)>1$.

{\bf Case 5}. $(s,i,t)=(7,4,3)$.
\begin{eqnarray*}
  T(n,k,7,4,3)&=&\frac{(n-k-3)(n-k-1)(7n-4k-26)(7n-6k-6)}{(n-6)^2(4n-k-17)(6n-5k-5)}\\&>&
  \frac{(n-k-3)(n-k-1)(7n-\frac{8}{3}k-26)(7n-\frac{22}{3}k-6)}{(n-6)^2(4n-k-17)(6n-5k-5)}
\end{eqnarray*}
Since $n\geq 4k-8$ and $k\geq 6$, we have
\begin{eqnarray*}
&&(n-k-1)(7n-\frac{8}{3}k-26)-(n-6)(6n-5k-5)\\&=&n^2-(4\frac{2}{3}k-8)n+(\frac{8}{3}k^2-\frac{4}{3}k-4)\\&\geq&
4k-4\geq 0
\end{eqnarray*}
and
\begin{eqnarray*}
&&(n-k-3)(7n-\frac{22}{3}k-6)-(n-6)(4n-k-17)\\&=&3n^2-(\frac{40}{3}k-14)n+(\frac{22}{3}k^2+22k+1)\\&\geq&
3(4k-8)^2-(\frac{40}{3}k-14)(4k-8)+(\frac{22}{3}k^2+22k+1)
\\&=&2k^2-\frac{22k}{3}-4>0.
\end{eqnarray*}
Hence $ T(n,k,7,4,3)>1$.
\end{proof}
\qed
\end{document}